\documentclass[12pt,twoside,reqno]{amsart}
\usepackage[colorlinks=true,citecolor=blue]{hyperref}
\usepackage{mathptmx, amsmath, amssymb, amsfonts, amsthm, mathptmx, enumerate, color,mathrsfs}
\usepackage{graphicx}

\usepackage{multirow}
\usepackage{epstopdf}
\usepackage{multicol}
\usepackage{algorithm}
\usepackage{algorithmic}
\usepackage{epstopdf}

\usepackage{mathtools}
\usepackage{amsmath,amssymb,amsfonts} 
\usepackage{dsfont}
\usepackage[numbers]{natbib}

\newcommand{\supp}{\mathop\mathrm{supp}}   
\newcommand{\dom}{\mathop\mathrm{dom}}   
\newcommand{\ssupp}{\mathop{\mathrm{\underline{supp}}}}  

\newcommand{\R}{\mathds{R}}

\newcommand{\CC}{\mathcal{C}}
\newcommand{\co}{\mathrm{co}}

\newcommand{\A}{\mathbf{A}}
\newcommand{\B}{\mathbf{B}}
\newcommand{\G}{\mathbf{G}}
\renewcommand{\H}{\mathbf{H}}

\newcommand{\F}{\mathscr{F}}

\newcommand{\x}{\mathbf{x}} 

\newtheorem{theorem}{Theorem}[section]

\newtheorem{proposition}[theorem]{Proposition}

\newtheorem{conjecture}[theorem]{Conjecture}

\theoremstyle{definition}
\newtheorem{definition}[theorem]{Definition}

\newtheorem{example}[theorem]{Example}
\newtheorem{remark}[theorem]{Remark}

\numberwithin{equation}{section}

\begin{document}
\setcounter{page}{1}

\title{Applications and Issues in Abstract Convexity}
\author[R. D\'iaz Mill\'an, N. Sukhorukova and J. Ugon]{R. D\'iaz Mill\'an$^{1}$, Nadezda Sukhorukova$^{2,*}$,  Julien Ugon$^{1}$}
\date{January 2022}

\vspace*{2.0cm}

\maketitle
\vspace*{-0.6cm}

\begin{center}
{\footnotesize

$^1$Faculty of Science, Engineering and Built Environment , Deakin University, Australia\\
$^2$Department of Mathematics, Swinburne University of Technology, Australia

}\end{center}

\vskip 4mm {\footnotesize \noindent {\bf Abstract.} The theory of abstract convexity, also known as convexity without linearity, is an extension of the classical convex analysis. There are a number of remarkable results, mostly concerning duality, and some numerical methods, however, this area has not found many practical applications yet. In this paper we study the application of abstract convexity to function approximation. Another important research direction addressed in this paper is the connection with the so-called axiomatic convexity.

 \noindent {\bf Keywords.} Abstract convexity, axiomatic convexity, Chebyshev approximation, Carath\'eodory number, quasiaffine functions.

 \noindent {\bf 2010 Mathematics Subject Classification.}  90C26,
 90C90,              
 90C47,            
 65D15.              
}

\renewcommand{\thefootnote}{}
\footnotetext{ $^*$Corresponding author.
\par
E-mail addresses:  r.diazmillan@deakin.edu.au (R. D\'iaz Mill\'an), nsukhorukova@swin.edu.au (N. Sukhorukova), julien.ugon@deakin.edu.au (Julien Ugon).
\par
Received ; Accepted . }

\maketitle

\section{Introduction}
Abstract convexity appears as a natural extension of the classical convex analysis when linear and affine functions (elementary functions) are ``replaced'' by other types of functions. This is why abstract convexity is also known as ``convexity without linearity''.

It was demonstrated that many results from the classical convex analysis: conjugation, duality, subdifferential related issues can be extended to ``non-linear'' settings~\citep{pallaschke1997Rolewicz,singerAbstract}.  In~\citep{pallaschke1997Rolewicz} the authors provide a very detailed historical review of the theory and it turned out that the origins of this environment comes back to early 70s, while some specific issues were already in place in the 50s of the twentieth century.

Despite a very productive work done in the development of abstract convexity, there are still several direction for improvements. From the onset of abstract convex theory, much effort has been devoted to study duality, leading to very elegant generalisations of classical convexity results. On the other hand, the geometric aspects of convexity have not received as much attention. Additionally, the theory remains ``abstract'' and needs practical applications. One of the goals of this paper is to provide some applications and develop a framework on the properties that the elementary functions have to satisfy to be efficient in applications.

There are a number of difficulties to develop efficient computational methods. The first step is to choose the set of elementary functions. This set should be simple enough to construct fast and efficient methods. On another hand, this set has to generate accurate approximations to the functions from the applications.  In some cases, the choice of elementary functions is reasonably straightforward (see section~\ref{sec:approximation}), but this is rather exceptional and in most applications the choice of a suitable class of elementary functions is a difficult task. 

A deep and detailed study on abstract convexity and related issues can be found in~\citep{singerAbstract} and also in~\citep{Rubinov00}. A detailed survey of  methods of abstract convex programming can be found in~\citep{AndramonovSurvey}. Many more contributions and developments and results on duality and methods are not mentioned in this paper, but the most recent state of the developments in abstract convexity and a comprehensive literature review can be found in~\citep{KrugerBuiBurachikYost}.

The paper is organised as follows. In section~\ref{sec:AbstractConvexity} we provide the essential background of abstract convexity and axiomatic convexity. In the same section, we underline connections and possible cross-feeding between these two types of convexity. Then, in section~\ref{sec:approximation} we illustrate how abstract convexity and axiomatic convexity can be applied to function approximation. Section~\ref{sec:numerical_experiments} illustrates the results of numerical experiments. Finally, section~\ref{sec:conc} provides conclusions and future research directions.

\section{Abstract Convexity}\label{sec:AbstractConvexity}

\subsection{Definitions and preliminaries}
Consider the set of all real-valued functions acting on the domain $X$, denoted by $\F$ and defined by $\F:=\{f:X\to \R\}$. We called the set of \emph{abstract linear functions} to any subset of functions $L\subseteq \F$. Following, we introduce some known definitions to state the abstract convexity concept. Suppose that we defined the set of abstract linear functions $L\subseteq \F$.


\begin{definition}
The vertical closure of $L$, $H := \{l+c: l\in L, c\in \R\}$ is the set of \emph{abstract affine functions}.
\end{definition}

  \begin{definition}[Abstract Convexity~\citep{Rubinov00}]\label{def:abstractConvex}
    A function $f$ is said to be \emph{$L$-convex} if there exists a set $U\subset H$ such that for any $x\in X$, $f(x) = \sup_{u\in U} u(x)$. Such functions are also called \emph{abstract convex with respect to $L$}.
  \end{definition}
  In this paper we denote the set of all $L$-convex functions by $\F_L$.

  \begin{definition}[Support Set~\citep{Rubinov00}]\label{def:support_set}
  The \emph{$L$-support set} of a function $f\in F$ is defined as:\\
  $ \supp_{L} f := \{l\in L: l\leq f\},$ 
  where $l\leq f$ means that $l(x)\leq f(x), \forall x\in X$.  The support sets of $L$-convex functions are called \emph{$L$-convex sets}.
  \end{definition}
  

Therefore, informally, convexity without linearity simply means that the ``role'' of linear functions of the classical convex analysis is held by functions from a certain class. The first (and for some applications this is the most important) question is how to choose the set of elementary function.

The next definition, which will be detailed and used in the next subsection, ``helps" the chosen of a ``better" set of abstract linear functions. 

\begin{definition}[Supremal generator~\citep{dutta2005abstract, Rubinov00}]
Let $\F$  be  a  set  of  functions defined  on  a  set~$X$.    A  set $L\subset \F$ is called  a \emph{supremal  generator} of $\F$ if  each  $f\in \F$ is abstract  convex with respect  to~$L$.
\end{definition}

\subsection{Quasiconvexity}

The notion of quasiconvexity plays an essential role in abstract convexity. We are now going to explain why.

The notion of quasiconvexity was originally introduced in~\citep{deFinetti1949}, where the author studied the behaviour of functions with convex sublevel sets, but the term quasiconvexity was introduced much later. 
\begin{definition}
 Let $D$ be a convex subset of $\R^n$. A function $ f:D\to \R$ is \emph{quasiconvex} if and only if its  sublevel set
	$$ S_{\alpha}=\{x\in D\, |\, f (x ) \leq \alpha \} $$
is  convex for any $\alpha \in \mathbb{R}$. 
\end{definition}
There are several equivalent definitions of quasiconvex functions, but the one we use is convenient for our study. 

\begin{definition}
A function $f:D\to \R$, where $D$ is a convex subset of $\R^n$,  is called quasiconcave if $-f$ is quasiconvex.
\end{definition}
\begin{definition}
Functions that are both quasiconvex and quasiconcave are called quasiaffine (quasilinear). 
\end{definition}
Note the following important observations.
\begin{itemize}
    \item Quasiconvex functions do not need to be continuous. 
\item In the case of univariate functions, quasiaffine functions are monotone functions.
\end{itemize}

 If a function is quasiaffine on~$\R^n$ (unconstrained problems) then,  from the definition, that its level sets must be half-spaces, and it is clear that the hyperplanes defining these half-spaces need to be parallel. In the presence of constraints, this observation is not valid and we provide an example in section~\ref{sec:numerical_experiments}.

 There are many studies dedicated to quasiconvex functions and quasiconvex  optimisation~\citep{SL,JPCrouzeix1980quasi,DaniilidisHadjisavvasMartinezLegas2002,dutta2005abstract,Rubinov00,RubinovSimsek}. In these studies, the notion of quasiconvexity appears as one of the possible generalisations of convexity.

\subsection{Quasiconvex functions and supremal generators}

Theorem~7.13 from~\citep{dutta2005abstract} states the following.
\begin{theorem}
 The set of all lower semicontinuous quasiaffine functions forms a supremal generator of the set of all lower semicontinuous quasiconvex functions.
\end{theorem}

Essentially, this theorem states that in the case of quasiconvex minimisation problems, quasiaffine functions ``replace'' the role of linear functions in classical convex settings. Therefore, the choice of elementary functions is clear (all quasiaffine functions).

One main advantage is that the sublevel sets of ``new'' elementary functions are half-spaces (similar to classical convex analysis). In most practical problems, there are some specific characteristics for the quasiaffine approximations (for example, the class of rational or generalised rational approximations, approximations that are compositions of certain monotone or non-decreasing function with certain affine or quasiaffine functions, etc.). This additional information helps dealing with some specific applications.

In the rest of this section, we clarify the ``roles'' of quasiaffine functions in the new abstract convex settings. Then section~\ref{sec:numerical_experiments} contains the results of computational experiments and illustrations.

\subsection{Relation to Axiomatic Convexity}

\emph{Axiomatic Convexity} aims at generalising the notion of convexity using only set-theoretic definitions. Good reviews can be found in~\citep{soltan84,vandevel93}. Given a set $X$, a family $\CC$ of subsets of $X$ is called a \emph{convexity} if it follows the following axioms:

\begin{enumerate}
  \item $\emptyset$ and $V$ belong to $\CC$;
  \item Any arbitrary intersection of sets from $\CC$ is in $\CC$: if $D\subset \CC$, \[\bigcap_{A\in D} A \in \CC\]
  \item The nested union of sets from $\CC$ is in $\CC$: given an a family of sets $D\subset \CC$ that is totally ordered by inclusion then \[\bigcup_{A\in D} A \in \CC.\]
\end{enumerate}
The pair $(X,\CC)$ is called a \emph{convex structure} (or \emph{(abstract) convex space}).

These axioms (in particular axioms 1 and 2) ensure the existence of the convex hull of any set $S\subset X$, namely the smallest convex set containing $S$:
\[ \co_{\CC} S = \bigcap \{A\in \CC: S\subset A\}. \]

If $F$ is finite, then $\co_{\CC} F$ is called a \emph{polyhedron}.

Axiomatic Convexity, also called \emph{abstract convexity} by~\citet{vandevel93}, has found applications in combinatorial geometry~\citep{duchet87,levi51}, among others. A family $\CC$ that satisfies the first two axioms above is called a \emph{closure space}, or $\cap$-stable~\citep{fan63}.

There are clear connections between abstract convexity as defined by Definition~\ref{def:support_set} and axiomatic convexity. Indeed, the notion of abstract convex function (Definition~\ref{def:abstractConvex}) first appeared in~\citep{fan63} in the context of closure spaces. The families of sets considered in \citep{fan63} were the closure of the set of sublevel sets of a given family of functions.

Abstract convex functions (Definition~\ref{def:abstractConvex}) and Abstract convex sets (Definition~\ref{def:support_set}) were then formally introduced by~\citet{kutateladze72}, who showed the equivalence between these sets and the ones considered by \citet{fan63}. Below we formalise some  results on the correspondence between abstract convex sets and closure spaces. Let us start with noting that the support set of any (not necessarily $L$-convex) function is $L$-convex. Indeed, the support set of $f$ is also the support set of its $L$-convex envelope, which is convex.

\begin{proposition}\label{prop:closure_space}
  Consider a family of functions $L$. Then, the set of $L$-convex sets form a closure space. Furthermore, every closure space is isomorphic to such a set.
\end{proposition}

\begin{proof}
  Consider a family of functions $L$ and let $\CC$ be the set of $L$-convex sets. We first show that $L$ is closed with respect to inclusion. Indeed, let $D\subset \CC$ be an arbitrary set of $L$-convex sets. For each $A\in D$, let $f_A(x) = \sup_{l\in A} l(x)$ and consider the function $f(x) = \inf_{A\in D}f_A(x)$. We will show that $\cap_{A\in D} A = \supp f$, which implies that $\cap_{A\in D} A$ is $L$-convex (note that $\supp f$ is the support set of the lower convex envelope of $f$).

    \begin{align*}
      \supp f &= \{l\in L: l(x) \leq \inf_{A\in D}f_A(x) \forall x\in X \} \\&= \{l\in L: l(x)\leq f_A(x) \forall A\in D, x\in X\} \\&= \cap_{A\in D}\{l\in L: l(x)\leq f_A(x), \forall x\in X\} = \cap_{A\in D} A.
    \end{align*}

  The first axiom is clearly verified by $\CC$ from the fact that $\emptyset$ is the support set of $\sup_{l\in \emptyset} l$ and $L$ is the support set of $\sup_{l\in L} l$.

  The second part of the statement is obtained by considering the set $L$ of indicator functions of the sets in the closure space. More precisely, let $\CC$ be a closure space, and consider the set $L=\{i_A: A\in \CC\}$ where \[i_A(x) = \begin{cases}
    0 & \text{if } x\in A\\
    +\infty &\text{otherwise}
    \end{cases}\]
    Then, $L$ is also the set of $L$-convex functions. Indeed, let $D\subset \CC$. Define $U = \{i_A: A\in D\} \subset L$ and $f=\sup_{l\in U} l = \sup_{A\in D} i_A$. Then $f(x) = 0$ if and only if $x\in A$ for all $A\in D$. That is, $f= i_{\cap_{A\in D} A}$. The isomorphism between the set of support sets of functions from $L$ and $\CC$ is evident.
\end{proof}

Note that in practice it is not necessary to consider the entire set of indicator functions, but only the set of indicator functions of a \emph{basis} of $\CC$ (i.e., a set of functions whose closure is $\CC$), which corresponds to a suppremal generator of the corresponding abstract functions. It can also easily be seen that the set of domains of the $L$-abstract functions defined as in the above proof, as well as the set of 0-sublevel sets, is precisely $\CC$.

On the other hand, not all families of abstract convex sets are a convexity structure.

\begin{example}
  Let $L$ be the set of linear functions on $\R$, and consider the functions $f_\alpha$ defined by $f_\alpha(x) = \alpha \lvert x\rvert$, for $\alpha>0$. These functions are $L$-convex, and induce a nested sequence of $L$-convex sets $S_\alpha = \supp_L f_\alpha = \{x\to \beta x: \beta\in [-\alpha,\alpha]\}$.

  However, \[U = \bigcup_{\alpha < 1} S_\alpha = \{x\to \beta x: \beta \in (-1,1)\}\]
  is not a $L$-convex set. Indeed, $f_1 = \sup_{l\in U} l$ is the smallest $L$-convex function whose support set contains $U$, but the function $x\to x\in \supp f_1\setminus U$, which shows that $U$ cannot be the support set of $f_1$ and therefore of any $L$-convex function. This shows that the set of $L$-convex sets is not a convex structure.

\end{example}

Let us introduce the notion of strict support set, which will enable us to address this issue.

  The $L$-strict support set of a function $f$ is defined as:
  \[ \ssupp f = \{l\in L: l(x)<f(x), \forall x\in \dom(f)\}. \]

We define the family of sets $\CC_L$ as follows:

\begin{definition}[convexity extension]
\begin{align*}
    \CC_L^f &:=\{A\in L: \ssupp f\subseteq A\subseteq \supp f\}\\
   \CC_L &:= \bigcup_{f\in L} \CC_L^f.
\end{align*}
We call the set $\CC_L$ the \emph{convexity extension} of the set of support sets of $L$-convex functions.
\end{definition}

\begin{proposition}\label{prop:convexity_space}
  For any family of functions $L$, the convexity extension of $L$, $\CC_L$,  forms a convexity structure.
\end{proposition}

\begin{proof}
  Since $\CC_L$ contains all $L$-convex sets, it contains a closure space and therefore by Proposition \ref{prop:closure_space}, it contains the sets $\emptyset$ and $X$.

    To see why the second axiom is satisfied consider an arbitrary family of sets $D\subset \CC_L$.
 For each $A\in D$, there exists a $L$-function $f_A$ such that $\ssupp f_A\subseteq A\subseteq \supp f_A$. Define $f=\co_L\inf_{A\in D} f_A$. Since $A\subseteq \supp f_A$ for any $A$, we have that
 \[
 \bigcap_{A\in D} A \subset \bigcap_{A\in D} \supp f_A = \supp f,
 \]
 where the last equality was shown in the proof of Proposition~\ref{prop:closure_space}. Additionally, we have:

    \begin{align*}
      \ssupp f &= \{l\in L: l(x) < \inf_{A\in D}f_A(x) \forall x\in X \} \\
      &\subset \{l\in L: l(x)< f_A(x) \forall A\in D, x\in X\} \\
      &= \cap_{A\in D}\{l\in L: l(x)< f_A(x)\forall x\in X\} \subset \cap_{A\in D} A.
    \end{align*}

    Therefore $\ssupp f\subseteq \cap_{A\in D} A\subseteq \supp f$, we conclude that $\cap_{A\in D} A$ is in $\CC_L$.

 Finally, to see that the third axiom is also satisfied, let $D$ be an ordered family of sets from $\CC$. For each $A\in D$, we define $f_A$ as above. The nested nature of the sets in $D$ implies that for $A,A'$ in $D$, if $A\leq_D A'$, then \begin{equation}\label{eq:totalOrder}
  f_A\leq f_{A'}.
  \end{equation}

  Let $S = \cup_{A\in D} A$ and $f=\sup_{A\in D} f_A$.

  It is clear that $S\subset \supp f$, since for any $u\in S$ there exists $A\in D$ such that $u\in \supp f_A$. Then $u\leq f_A\leq f$. Now, consider $u\notin S$. Then $u\notin \cup_{A\in D} \ssupp f_A$, and so there exists $x\in X$ such that $u(x)\geq f_A(x)$ for any $A\in D$, and therefore that $u(x)\geq \sup_{A\in D} f_A(x) = f(x)$. Therefore $u\notin \ssupp f$. This implies that $\ssupp f\subseteq S\subseteq \supp f$, and therefore $S$ is in $\CC_L$ and $\CC_L$ is a convexity structure.
\end{proof}

\begin{remark}
  Propositions~\ref{prop:closure_space} and ~\ref{prop:convexity_space} imply an equivalence relation between sets of abstract linear functions according to the convexity structure (axioms 1-3) or closure space (axioms 1,2) the induced abstract convex sets are isomorphic to.
\end{remark}

Axiomatic convexity has been applied to obtain generalisations of well known geometric results in convexity theory. Of particular interest to this paper are generalisations of Carath\'eodory's theorems. Much research has been devoted to investigating this topic. We refer to \citep{vandevel93} for a review of classical results in the area.

\begin{definition}[{\citep[Definition 1.5]{vandevel93}}]
  In a given Convexity structure $\CC$,
  a set $F$ is \emph{Carath\'eodory dependent} if \[\co_{\CC} F\subseteq \bigcup_{a\in F} \co_{\CC} (F\setminus \{a\}),\] and \emph{Carath\'eodory independent} otherwise.

  Then we can define the \emph{Carath\'eodory number} of a convex structure $\CC$ as the largest cardinality  of a Carath\'eodory independent set (i.e., any set $F$ with cardinality greater than $c(x)$ is Carath\'eodory dependent). The Carath\'eodory numbers of several important classes of convexity structures are known and discussed in the literature~\citep{vandevel93}.

\end{definition}

Generalisations of Helly's and Radon's theorems were obtained similarly. The Carath\'eodory number enables the following generalisation of Carath\'eodory's classical result:

\begin{proposition}[{\citep[Theorem 1.7]{vandevel93}}]
  Consider a convexity structure $(\CC,X)$ with Carath\'eodory number $c(x)$, and let $A\subset \CC$ be a convex set from $\CC$. Then, for any $x\in A$ there exists a set $F\subset A$ such that $\lvert F\rvert \leq c(x)$ and $x\in \co F$. This is best possible.
\end{proposition}


\section{Application to approximation}\label{sec:approximation}
\subsection{Problem formulation}

We are working with uniform (Chebyshev) approximation:
\begin{equation}\label{eq:cheb_formulation}
\min_A \sup_{\x\in X}|f(\x)-g(A,\x)|,
\end{equation}
where $A$ are the decision variables and $X$ is a  compact set. In many computer based applications, this set is a finite grid defined on a convex compact set. If we define $L = \{A\to g(A,\x): x\in X\}$, then it is clear from Formulation~\eqref{eq:cheb_formulation} that the objective function of this problem is $H_L$-convex.


In the case when the approximation function~$g(A,t)$ is a polynomial (and its coefficients are subject to optimisation), the optimality conditions are based on maximal deviation points (Chebyshev theorem, see~\citep{chebyshev}). In the case of univariate function approximation, the conditions are based on the notion of alternating sequence.

\begin{theorem}[Chebyshev]\label{thm:Chebyshev}
A polynomial of degree at most~$n$ is optimal in uniform (Chebyshev) norm  if and only there are $n+2$ alternating points. 
\end{theorem}

Essentially, the approximations here are presented as linear combinations of basis functions. In the case of polynomials the basis functions are monomials, but other classes of basis functions can be used. The corresponding optimisation problems are convex as supremum of linear forms. Theorem~\ref{thm:Chebyshev} can be proved using convex analysis and the number of alternative point ($n+2$) obtained through Carath\'eodory's theorem. In other words, the following conjecture is true, in the case when $L$ is a set of linear functions:
\begin{conjecture}\label{conjecture:caratheodori}
  For a function from $L$ to be an optimal approximation in uniform norm, it is enough that it is optimal at $c(\CC_{H_L})+1$ extreme points.
\end{conjecture}
We also conjecture that this is best possible, in the sense that the statement is generally not true if we replace $c(\CC_{H_L})+1$ with $c(\CC_{H_L})$.

If true, this conjecture could form the basis for algorithms for best approximation, such a generalisation of Vall\'ee-Poussin's procedure~\citep{valleepoussin:1911}. Later in this section,  we discuss an  example of families $L$ in the context of this conjecture.

One possible generalisation to polynomial approximation is rational approximation, that is, approximation by the ratio of two polynomials, whose coefficients are subject to optimisation. Note that in some cases, the degree of the denominator and numerator can be reduced without compromising the accuracy:
$$\frac{\sum_{j=0}^n a_jt_i^j}{\sum_{k=0}^{m}b_kt_i^k}=\frac{\sum_{j=0}^{n-\nu} a_jt_i^j}{\sum_{k=0}^{m-\mu}b_kt_i^k},$$
where $d=\min\{\nu,\mu\}$  is called \textit{the defect}. Then the necessary and sufficient optimality conditions are as follows~\cite{achieser1956}.
\begin{theorem} \label{thm:equioscillation_characterization}
A rational function in $\mathcal{R}_{n,m}$ with defect $d$ is the best polynomial approximation of a function $f \in C^0(I)$ (the space of continuous functions over $I$) if and only if there exists a sequence of at least $n+m+2-d$ points of maximal deviation where the sign of the deviation at these points alternates.
\end{theorem}
 Therefore, similar to polynomial approximation, these conditions are bases the number of alternating points. 

In connection to Conjecture~\ref{conjecture:caratheodori}, $c(\CC_{H_L})$ is $n+m+2$. If $n$ and $m$ are the degrees of the corresponding polynomials, then the total number of the parameters is $n+1+m+1=n+m+2$. However, one of the parameters can be fixed (otherwise one can divide the numerator and denominator by the same number). Therefore, there are $n+m+1$ ``free'' parameters and we add one more for~ $c(\CC_{H_L})$ (similar to Carath\'eodory's theorem from classical convex analysis). We will provide a formal proof in a future paper.

Rational approximation was a very popular research area in the 1950s-70s~\citep{achieser1956,Boehm1964,Meinardus1967rational,Ralston1965Reme,Rivlin1962} (just to name a few).

If the basis functions are not restricted to monomials, the approximations are called generalised rational approximations. The term is due to Cheney and Loeb~\citep{cheney1964generalized}.

One way to approach rational approximation is through constructing   ``near optimal'' solutions~\citep{Trefethen2018}. This approach is very efficient and therefore very popular. The extension of this approach to non-monomial basis functions and multivariate approximation remains open (an extension to complex domains can be found in~\citep{Trefethen2020}). 

An alternative way is based on modern optimisation techniques. This approach is preferable when the basis function are not restricted to univariate monomials and when deeper theoretical study is required. The corresponding optimisation problems are quasiconvex and can be solved using a general quasiconvex optimisation methods.

Rational and generalised rational functions are quasiaffine as ratios of linear forms~\citep{SL,loeb1960,amcpeirissukhsharonugon}. It can be solved efficiently by applying a simple, but robust approach, called the bisection method for quasiconvex functions~\citep{SL}.

In~\citep{amcpeirissukhsharonugon} the authors use a well-known bisection method for quasiconvex functions (see \citep{SL}) to solve these problems. In \citep{diazsukhorugon}, the authors use a projection-type algorithm for solving these problems. 

\subsection{Bisection method}

Bisection method for quasiconvex functions~\citep{SL} can be used efficiently for rational and generalised rational   approximation, including multivariate settings~\citep{MultivariateRational, amcpeirissukhsharonugon}. Can we use abstract convexity to extend this method to a wider class of approximations?

 The problem can be formulated as follows:
\begin{equation}\label{eq:problem_obj}
\text{minimise}~\tilde{z}
\end{equation}
subject to
\begin{equation}\label{eq:problem_con1}
f(\x)-\frac{\A^T\G(\x)}{\B^T\H(\x)}\leq \tilde{z},~\x\in X,
\end{equation}
\begin{equation}\label{eq:problem_con2}
\frac{\A^T\G(\x)}{\B^T\H(\x)}-f(\x)\leq \tilde{z},~\x\in X,
\end{equation}
\begin{equation}\label{eq:positivity}
\B^T\H(\x)\geq\delta,~\x\in X,
\end{equation}
where $\tilde{z}$ is the maximal deviation and $\delta$ is a small positive number. Problem~(\ref{eq:problem_obj})-(\ref{eq:positivity}) is not linear.

The idea of this method is based on the fact that all the sublevel sets of quasiconvex functions are convex and therefore the sublevel sets of quasiaffine functions are half-spaces. Essentially, this means that the constraint set~(\ref{eq:problem_con1})-(\ref{eq:positivity}) for a fixed~$\tilde{z}$ is an intersection of a finite number of half-spaces ($X$ is a finite grid) and therefore it is a polytope.  

The algorithm starts with an upper and lower bound ($u$ and $l$) for the maximal deviation, then the sublevel set for the maximal deviation at the level~$z=\frac{u+l}{2}$ is a convex set and the algorithm checks if this set is empty. If it is empty, then the upper bound  is updated to~$z$, otherwise, the lower bound is updated to~$z$. The algorithm terminates when the upper and lower bounds are within the specified precision.

In general, checking if the convex set (sublevel set of the maximal deviation) is empty or not may be a difficult task (convex feasibility problems). There are a number of efficient methods (\citep{BauschkeLewis, Shi-yaXu, YangYang, Zaslavski, Zhao} just to  name a few), but there are still several open problems here. The discussion of  these problems  is out of scope of the current paper.

In the case of multivariate generalised rational approximation, however, this problem (convex feasibility) can be reduced to solving a linear programming problem~\citep{MultivariateRational}.  The denominator of the approximation does not change the sign. Assume, for simplicity, that it is positive, then the problem of checking the feasibility  is equivalent to solving the following problem:
\begin{equation}\label{eq:problemLP_obj_ax}
\text{minimise}~\tilde{u}
\end{equation}
subject to
\begin{equation}\label{eq:problemLP_con_ax1}
f(\x){\B^T\H(\x)}-{\A^T\G(\x)}\leq z{\B^T\H(\x)}+\tilde{u},~\x\in X,
\end{equation}
\begin{equation}\label{eq:problemLP_con_ax2}
{\A^T\G(\x)}-f(\x){\B^T\H(\x)}\leq z{\B^T\H(\x)}+\tilde{u},~\x\in X,
\end{equation}
\begin{equation}\label{eq:positivityLPax}
\B^T\H(\x)\geq\delta,~\x\in X,
\end{equation}
where $z=\frac{1}{2}(u+l)$ is the current bisection point (bisecting the possible values of maximal deviation). 

If an optimal solution $\tilde{u}\leq 0$, the corresponding sublevel set of the maximal deviation function is not empty (update the upper bound), otherwise the set is empty (update the lower bound). If $X$ is a finite grid, then (\ref{eq:problemLP_obj_ax})-(\ref{eq:positivityLPax}) is a linear programming problem and can be solved efficiently at each step of the bisection method.

To summarise, an efficient implementation of the bisection method for quasiconvex functions,  can be extended to approximation by any type of quasiaffine functions, since at each step one needs to check if a polytope (intersection of a finite number of half-spaces) is non-empty. The main problem is how to find this polytope. In the case of rational and generalised rational approximation, this problem is simple, but not for some other types of quasiaffine approximations.  In section~\ref{sec:numerical_experiments} we give more examples where the construction of this polytope is straightforward (composition of monotone and affine or ratios of affine functions). 

One possibility for constructing sublevel sets for smooth quasiaffine approximation is to use their gradients (if it is not vanishing). In the rest section, we give an example demonstrating that this approach may be complicated if there are constraints (even simple linear constraints). In particulra, this approach will not work even in the case of rational approximation due to the requirement for the denominator to be strictly positive.

\subsection{Approximation by a quasiaffine function}\label{sub:ApproxByQuasiaffine}

Suppose that instead of a generalised rational approximation one needs to approximate by a function from a different class, but the approximations are quasiaffine functions with respect to the decision variables. Then the problem is as follows:

\begin{equation}\label{eq:problem_obj_quasi}
\text{minimise}~\tilde{z}
\end{equation}
subject to
\begin{equation}\label{eq:problem_con1_quasi}
f(\x)-g(A,\x)\leq \tilde{z},~\x\in X,
\end{equation}
\begin{equation}\label{eq:problem_con2_quasi}
g(A,\x)-f(\x)\leq \tilde{z},~\x\in X,
\end{equation}
where $g(A,\x)$ is a quasiaffine function. Since the sublevel sets of quasiaffine functions are half-spaces, the constraint set~(\ref{eq:problem_con1_quasi})-(\ref{eq:problem_con2_quasi}) is a polytope, since $X$ is a finite grid.

\begin{remark}
A finite number of linear constraints may be added to~(\ref{eq:problem_obj_quasi})-(\ref{eq:problem_con2_quasi}), while the constraint set remains a polytope. 
\end{remark}

Therefore, the constraint set (for any fixed $\tilde{z}$), with or without additional linear constraints, is a poytope and we only need to check if this polytope is non-empty. However, we still need an efficient approach for finding this polytope. In the case of smooth functions, one can use the gradient as a possible normal vector to the hyperplanes, providing that it is not a zero vector. In section~\ref{sec:numerical_experiments} we study the examples where the quasiaffine approximations can be decomposed as a composition of strictly monotone and affine (quasiaffine) functions. Other situations may be harder.

Note that if a function is quasiaffine on the whole space, then the sublevel sets are half-spaces, whose boundary hyperplanes  are parallel. Otherwise, two hyperplanes will intersect. If there are additional (even linear) constraints, this observation may not be valid anymore. The following example illustrates that this may happen even in a very simple case of two variables. 

\begin{example}
Consider $f(x,y)=\frac{x}{y}$, where $y>0$. The sublevel sets are
$$S_{\alpha}=\{(x,y): \frac{x}{y}\leq \alpha, y>0\},$$
where $\alpha$ is a given real number. Then $S_{\alpha}$ can be described as 
$$\{x-\alpha y\leq 0,\quad y>0\}.$$
Each sublevel set is still a half-space, but the corresponding hyperplanes (in this example they correspond to level sets) are not parallel. These hyperplanes intersect at $(0,0)$, but this point is excluded from the domain due to the requirement for the denominator~$y$ to be strictly positive.

\end{example}

This example demonstrates that the computation of sublevel sets for quasiaffine functions may be challenging. In the next section we give some examples of the classes of approximations that can be handled efficiently. We also provide the results of the numerical experiments.

\section{Numerical experiments}\label{sec:numerical_experiments}

\subsection{Strictly monotone function in composition with an affine function.}
In this example we approximate function
\begin{equation}\label{eq:f_example_affine}
f(x,y)=(-x+y^3+x^4)^4
\end{equation}
by a quasiaffine function in the form
\begin{equation}\label{eq:g_example_affine}
g(A,x,y)=(a_1+a_2x+a_3y+a_4x^2+a_5y^2+a_6xy)^3,
\end{equation}
where $A=(a_1,a_2,a_3,a_4,a_5,a_6)$ are the decision variables, $x$ and $y$ form a grid on $[-1,1]$ with the step-size~0.1. The optimal coefficients are (rounded to two decimal places):
$$A=(-1.88,   -0.75,   0.31,    3.29,    0.98,   -0.74),$$
the maximal absolute deviation is~8.01 (to two decimal places).

Figure~\ref{fig:fun} represents the function~$f(x,y)$, figure~\ref{fig:approx1} depicts the best found approximation and figure~\ref{fig:dev1} contains the deviation function $f(x,y)-g(A,x,y)$, which is the error of approximation.

Function $f(x,y)$ is almost flat with an abrupt increase around point~$(-1,1)$. Visually, the approximation resembles the shape of the original function and the magnitude of the deviation confirms that the approximation is reasonably good.  

\begin{figure}[h]
\centering 
\includegraphics[width=0.8\textwidth]{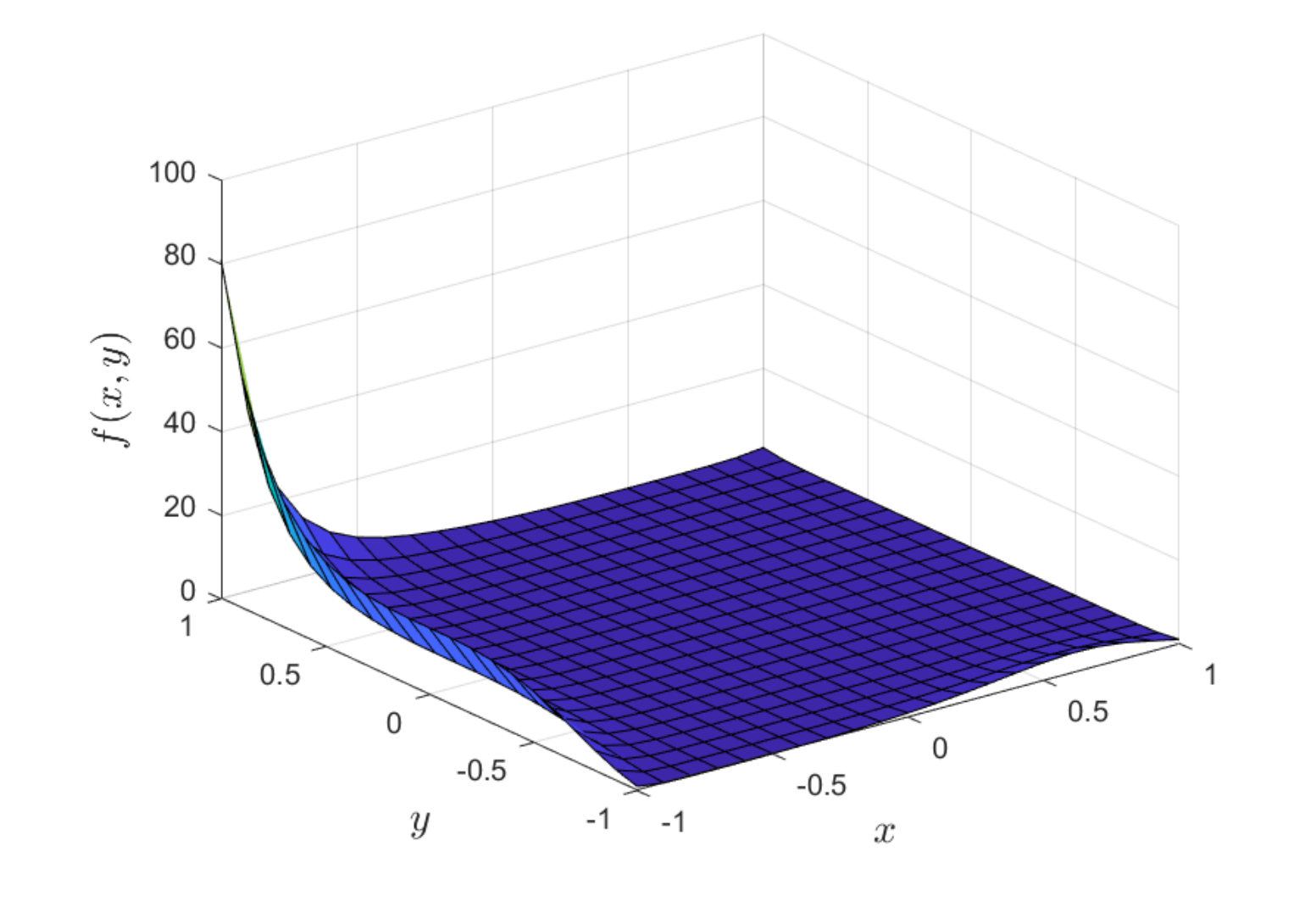}
\caption{Function $f(x,y)=(-x+y^3+x^4)^4$.}
\label{fig:fun}
\end{figure}

\begin{figure}[h]
\centering 
\includegraphics[width=0.8\textwidth]{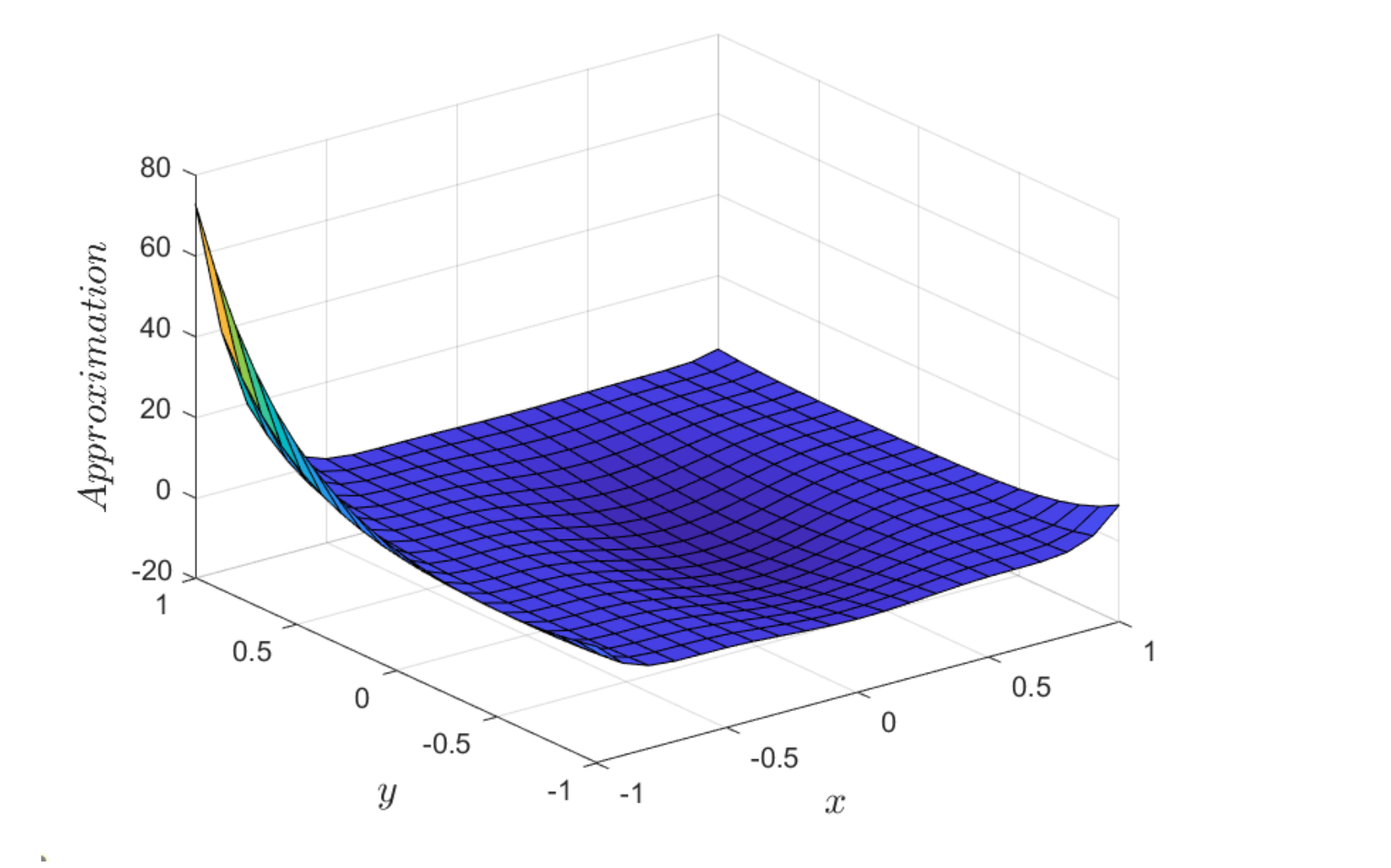}
\caption{Approximation $g(A,x,y)=(a_1+a_2x+a_3y+a_4x^2+a_5y^2+a_6xy)^3$.}
\label{fig:approx1}
\end{figure}

\begin{figure}[h]
\centering 
\includegraphics[width=0.8\textwidth]{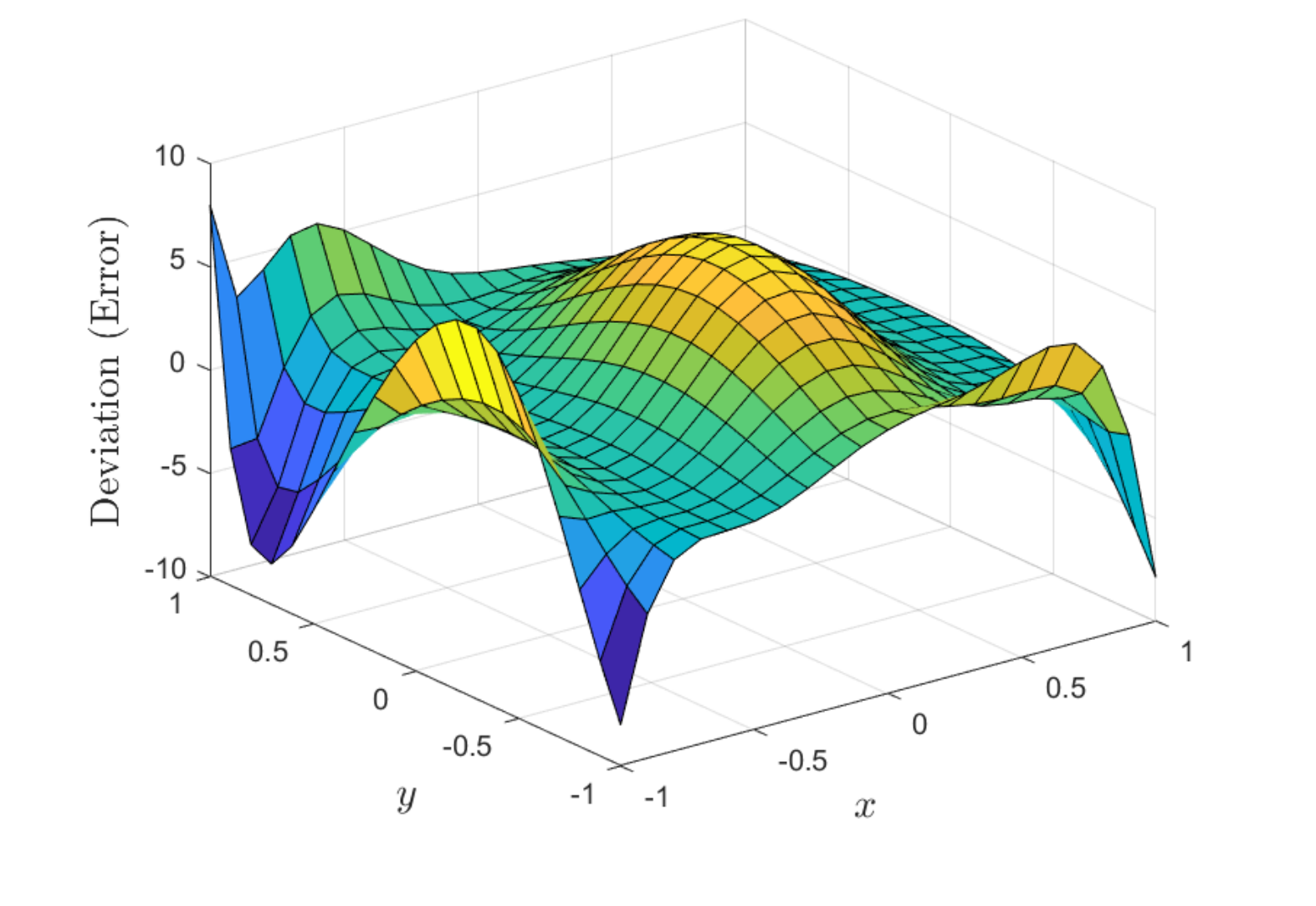}
\caption{Deviation (error) function  $$f(x,y)-g(A,x,y)=(-x+y^3+x^4)^4-(a_1+a_2x+a_3y+a_4x^2+a_5y^2+a_6xy)^3.$$}
\label{fig:dev1}
\end{figure}

\subsection{Strictly monotone function in composition with a rational function.}
In this example we approximate the same  function
\begin{equation}\label{eq:f_example_affine}
f(x,y)=(-x+y^3+x^4)^4
\end{equation}
by a quasiaffine function in the form
\begin{equation}\label{eq:g_example_affine}
g_1(A,x,y)=\left(\frac{a_1+a_2x+a_3y+a_4x^2+a_5y^2}{1+a_6xy}\right)^3,
\end{equation}
where, similar to the previous example, $A=(a_1,a_2,a_3,a_4,a_5,a_6)$ are the decision variables (the number of decision variables is the same as in the previous example), $x$ and $y$ are on a grid  $[-1,1]$ with the step-size~0.1. The purpose of approximation is to use lower degree polynomials in the composition. The optimal coefficients are (rounded to two decimal places):
$$A=(-1.66,   -0.93,   1.05,   2.87,  0.98,  1),$$
the maximal absolute deviation is~11.48 (to two decimal places).

Figure~\ref{fig:approx2} contains the best found approximation and figure~\ref{fig:dev2} depicts the deviation function $f(x,y)-g_1(A,x,y)$, which is the error of approximation.

The new approximation still resembles the shape of the original function, but the maximal deviation is higher.

\begin{figure}[h]
\centering 
\includegraphics[width=0.8\textwidth]{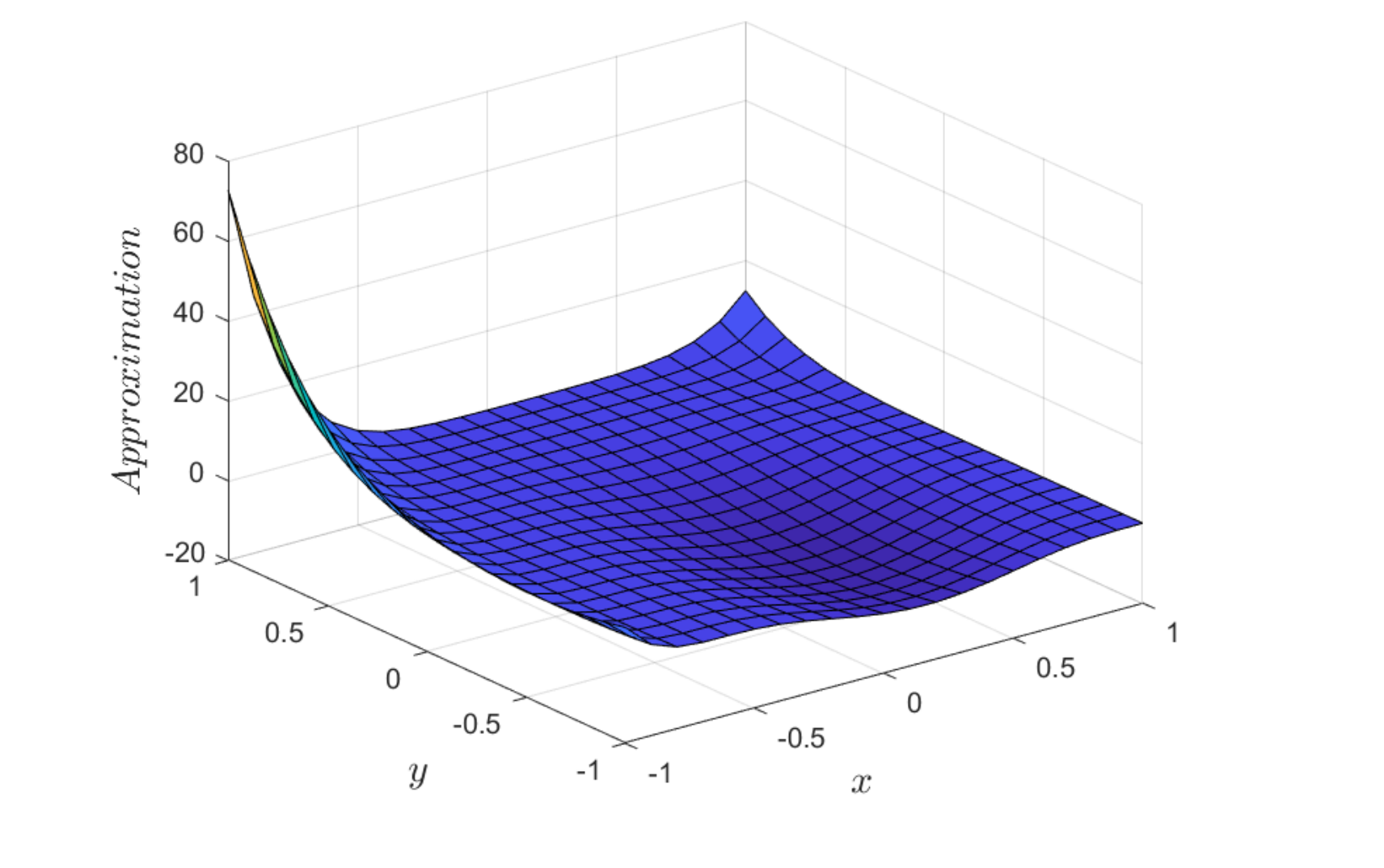}
\caption{Approximation $g_1(A,x,y)=(\frac{a_1+a_2x+a_3y+a_4x^2+a_5y^2}{1+a_6xy})^3$.}
\label{fig:approx2}
\end{figure}

\begin{figure}[h]
\centering 
\includegraphics[width=0.8\textwidth]{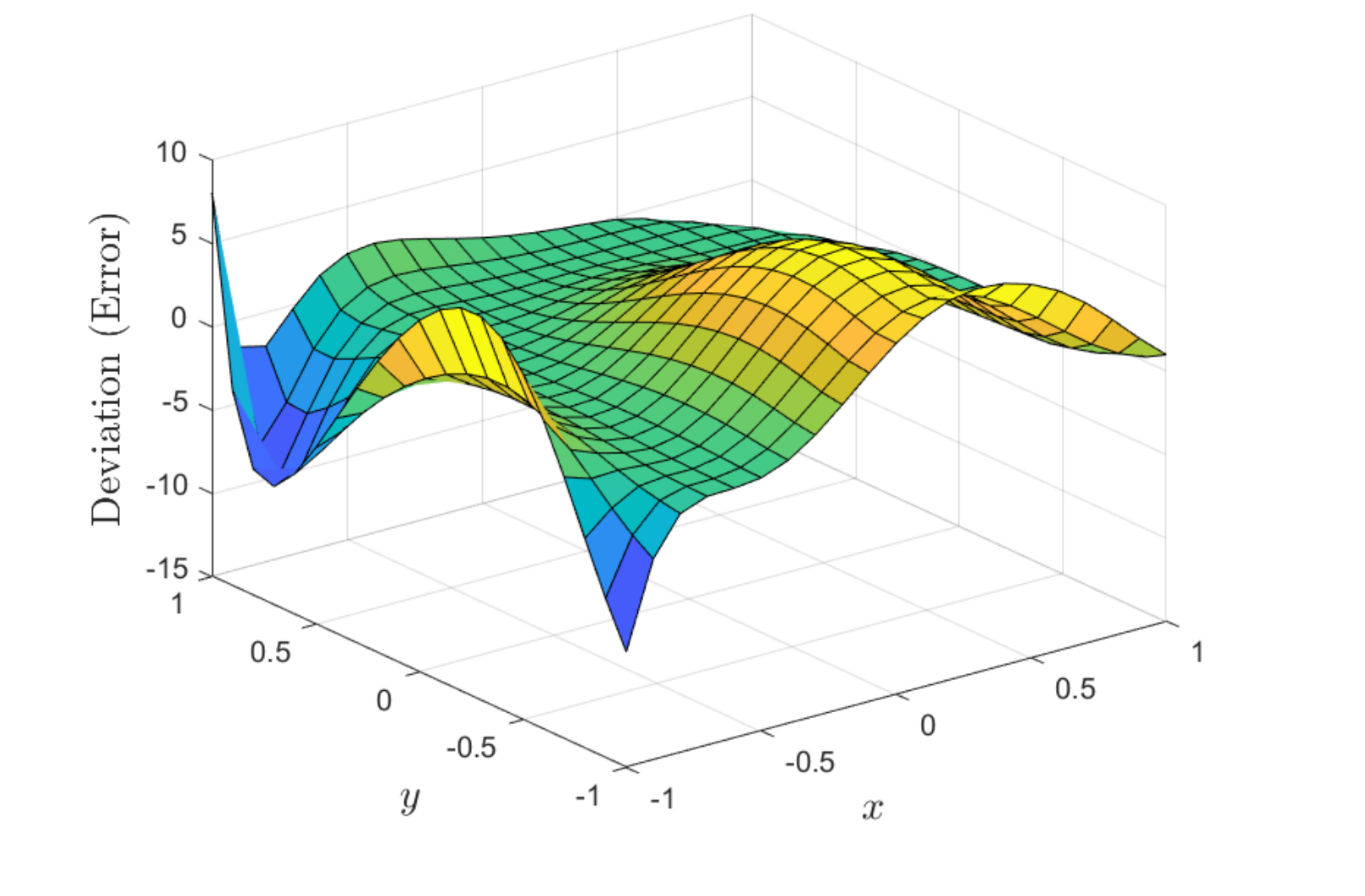}
\caption{Deviation (error) function  $$f(x,y)-g(A,x,y)=(-x+y^3+x^4)^4-\left(\frac{a_1+a_2x+a_3y+a_4x^2+a_5y^2}{1+a_6xy}\right)^3.$$}
\label{fig:dev2}
\end{figure}
\subsubsection{Deep learning applications.}

At first glance, the class of a composition of a monotone univariate function with an affine  or rational function is restrictive. However, it has many practical applications, including deep learning, where compositions of univariate activation functions and affine transformations are the main components or the models~\citep{Goodfellow2016}. Most common choices for activation functions are sigmoid functions, ReLU and Leaky ReLU (the last two functions are simply non-decreasing piecewise linear functions with two linear pieces).   

\section{Conclusions}\label{sec:conc}
The goal of this study is to demonstrate that abstract convexity (in the sense ``convexity without linearlity'') has several applications in different branches of mathematics, including approximation theory. This application appears naturally in the new settings.

The results of the numerical experiments demonstrate that our approach is computationally efficient. The applications lead to practical applications as well. One potential application is data science and deep learning.

We also touched the connections between ``abstract convexity'' and ``axiomatic convexity''. These areas have many overlappings that can be seen as a new way of looking at the problems, in particular,  function approximation.
\vskip 6mm
\noindent{\bf Acknowledgements}

\noindent
This research was supported by the Australian Research Council (ARC),  Solving hard Chebyshev approximation problems through nonsmooth analysis (Discovery Project DP180100602).

\newpage

\bibliographystyle{plainnat}
\bibliography{references}

\begin{thebibliography}{37}
\providecommand{\natexlab}[1]{#1}
\providecommand{\url}[1]{\texttt{#1}}
\expandafter\ifx\csname urlstyle\endcsname\relax
  \providecommand{\doi}[1]{doi: #1}\else
  \providecommand{\doi}{doi: \begingroup \urlstyle{rm}\Url}\fi

\bibitem[Achieser(1956)]{achieser1956}
Achieser.
\newblock \emph{Theory of Approximation}.
\newblock Frederick Ungar, New York, 1956.

\bibitem[Andramonov(2002)]{AndramonovSurvey}
Mikhail Andramonov.
\newblock A survey of methods of abstract convex programming.
\newblock \emph{Journal of Statistics and Management Systems}, 5\penalty0
  (1-3):\penalty0 21--37, 2002.
\newblock \doi{10.1080/09720510.2002.10701049}.
\newblock URL \url{https://doi.org/10.1080/09720510.2002.10701049}.

\bibitem[Bauschke and Lewis(2000)]{BauschkeLewis}
Heinz~H Bauschke and Adrian~S Lewis.
\newblock Dykstras algorithm with bregman projections: A convergence proof.
\newblock \emph{Optimization}, 48\penalty0 (4):\penalty0 409--427, 2000.

\bibitem[Boehm(1964)]{Boehm1964}
Boehm.
\newblock Functions whose best rational chebyshev approximation are
  polynomials.
\newblock \emph{Numer. Math.}, pages 235–--242, 1964.

\bibitem[Boyd and Vandenberghe(2009)]{SL}
S.~Boyd and L.~Vandenberghe.
\newblock \emph{Convex Optimization}.
\newblock Cambridge University Press, New York, USA, seventh edition, 2009.

\bibitem[Bui et~al.(2021)Bui, Burachik, Kruger, and
  Yost]{KrugerBuiBurachikYost}
Hoa~T. Bui, Regina~S. Burachik, Alexander~Y. Kruger, and David~T. Yost.
\newblock Zero duality gap conditions via abstract convexity.
\newblock \emph{Optimization}, 0\penalty0 (0):\penalty0 1--37, 2021.
\newblock \doi{10.1080/02331934.2021.1910694}.
\newblock URL \url{https://doi.org/10.1080/02331934.2021.1910694}.

\bibitem[Chebyshev(1955)]{chebyshev}
P.L Chebyshev.
\newblock The theory of mechanisms known as parallelograms.
\newblock In \emph{Selected Works}, pages 611--648. Publishing Hours of the
  USSR Academy of Sciences, Moscow, 1955.
\newblock (In Russian).

\bibitem[Cheney and Loeb(1964)]{cheney1964generalized}
E.W. Cheney and H.L. Loeb.
\newblock Generalized rational approximation.
\newblock \emph{Journal of the Society for Industrial and Applied Mathematics,
  Series B: Numerical Analysis}, 1\penalty0 (1):\penalty0 11--25, 1964.

\bibitem[Crouzeix(1980)]{JPCrouzeix1980quasi}
J.~P. Crouzeix.
\newblock Conditions for convexity of quasiconvex functions.
\newblock \emph{Mathematics of Operations Research}, 5\penalty0 (1):\penalty0
  120--125, 1980.

\bibitem[Daniilidis et~al.(2002)Daniilidis, Hadjisavvas, and
  Martinez-Legaz]{DaniilidisHadjisavvasMartinezLegas2002}
A.~Daniilidis, N.~Hadjisavvas, and J.-E. Martinez-Legaz.
\newblock An appropriate subdifferential for quasiconvex functions.
\newblock \emph{SIAM Journal on Optimization}, 12\penalty0 (2):\penalty0
  407--420, 2002.

\bibitem[de~Finetti(1949)]{deFinetti1949}
B.~de~Finetti.
\newblock Sulle stratificazioni convesse.
\newblock \emph{Ann. Mat. Pura Appl.}, pages 173--183, 1949.

\bibitem[de~la Vall{\'e}e~Poussin()]{valleepoussin:1911}
C~J de~la Vall{\'e}e~Poussin.
\newblock Sur la m{\'e}thode de l'approximation minimum.
\newblock \emph{Ann. Soc. Sci. Bruxelles}, 35:\penalty0 1--16.

\bibitem[Duchet(1987)]{duchet87}
Pierre Duchet.
\newblock Convexity in combinatorial structures.
\newblock In \emph{Proceedings of the 14th Winter School on Abstract Analysis},
  pages [261]--293. Circolo Matematico di Palermo, 1987.
\newblock URL \url{http://eudml.org/doc/220736}.

\bibitem[Dutta and Rubinov(2005)]{dutta2005abstract}
J.~Dutta and A.~M. Rubinov.
\newblock Abstract convexity.
\newblock \emph{Handbook of generalized convexity and generalized
  monotonicity}, 76:\penalty0 293--333, 2005.

\bibitem[Fan(1963)]{fan63}
Ky~Fan.
\newblock On the {K}rein-{M}ilman theorem.
\newblock In \emph{Convexity: Proceedings of the Seventh Symposium in Pure
  Mathematics of the American Mathematical Society}, pages 211--219. American
  Mathematical Society, 1963.
\newblock \doi{10.1090/pspum/007/0154097}.
\newblock URL \url{https://doi.org/10.1090%2Fpspum%2F007%2F0154097}.

\bibitem[Goodfellow et~al.(2016)Goodfellow, Bengio, and
  Courville]{Goodfellow2016}
Ian Goodfellow, Yoshua Bengio, and Aaron Courville.
\newblock \emph{Deep Learning}.
\newblock MIT Press, 2016.
\newblock \url{http://www.deeplearningbook.org}.

\bibitem[Kutateladze and Rubinov(1972)]{kutateladze72}
S~S Kutateladze and A~M Rubinov.
\newblock Minkowski duality and its applications.
\newblock \emph{Russ. Math. Surv.}, 27\penalty0 (3):\penalty0 137--191, jun
  1972.
\newblock \doi{10.1070/rm1972v027n03abeh001380}.
\newblock URL \url{https://doi.org/10.1070%2Frm1972v027n03abeh001380}.

\bibitem[Levi(1951)]{levi51}
F.~W. Levi.
\newblock On helly's theorem and the axioms of convexity.
\newblock \emph{Journal of the Indian Mathematical Society}, 15:\penalty0
  65--76, 1951.

\bibitem[Loeb(1960)]{loeb1960}
Henry~L Loeb.
\newblock Algorithms for chebyshev approximations using the ratio of linear
  forms.
\newblock \emph{Journal of the Society for Industrial and Applied Mathematics},
  8\penalty0 (3):\penalty0 458--465, 1960.

\bibitem[Meinardus(1967)]{Meinardus1967rational}
G.~Meinardus.
\newblock \emph{Approximation of Functions: Theory and Numerical Methods}.
\newblock Springer-Verlag, Berlin, 1967.

\bibitem[Mill\'an et~al.({\natexlab{a}})Mill\'an, Peiris, Sukhorukova, and
  Ugon]{MultivariateRational}
R.~D\'iaz Mill\'an, V.~Peiris, N.~Sukhorukova, and J.~Ugon.
\newblock Multivariate approximation by polynomial and generalised rational
  functions.
\newblock \emph{Accepted in Optimization}, {\natexlab{a}}.
\newblock URL \url{https://arxiv.org/abs/2101.11786v1}.

\bibitem[Mill\'an et~al.({\natexlab{b}})Mill\'an, Sukhorukova, and
  Ugon]{diazsukhorugon}
R.~D\'iaz Mill\'an, N.~Sukhorukova, and J.~Ugon.
\newblock An algorithm for best generalised rational approximation of
  continuous functions.
\newblock \emph{In PRESS, accepted in Set-Valued and Variational Analysis},
  {\natexlab{b}}.
\newblock URL \url{https://arxiv.org/abs/2011.02721}.

\bibitem[Nakatsukasa and Trefethen(2020)]{Trefethen2020}
Yuji Nakatsukasa and Lloyd~N. Trefethen.
\newblock An algorithm for real and complex rational minimax approximation.
\newblock \emph{SIAM Journal on Scientific Computing}, 42\penalty0
  (5):\penalty0 A3157--A3179, 2020.
\newblock \doi{10.1137/19M1281897}.

\bibitem[Nakatsukasa et~al.(2018)Nakatsukasa, Sete, and
  Trefethen]{Trefethen2018}
Yuji Nakatsukasa, Olivier Sete, and Lloyd~N. Trefethen.
\newblock The aaa algorithm for rational approximation.
\newblock \emph{SIAM Journal on Scientific Computing}, 40\penalty0
  (3):\penalty0 A1494--A1522, 2018.

\bibitem[Pallaschke and Rolewicz(1997)]{pallaschke1997Rolewicz}
Diethard Pallaschke and Stefan Rolewicz.
\newblock \emph{Foundations of Mathematical Optimization: Convex Analysis
  without Linearity}, volume 388.
\newblock Springer Science \& Business Media, 1997.

\bibitem[Peiris et~al.(2021)Peiris, Sharon, Sukhorukova, and
  Ugon]{amcpeirissukhsharonugon}
V.~Peiris, N.~Sharon, N.~Sukhorukova, and J.~Ugon.
\newblock Generalised rational approximation and its application to improve
  deep learning classifiers.
\newblock \emph{Applied Mathematics and Computation}, 389:\penalty0 125560, jan
  2021.
\newblock \doi{10.1016/j.amc.2020.125560}.
\newblock URL \url{https://doi.org/10.1016%2Fj.amc.2020.125560}.

\bibitem[Ralston(1965)]{Ralston1965Reme}
Anthony Ralston.
\newblock Rational chebyshev approximation by remes' algorithms.
\newblock \emph{Numer. Math.}, 7\penalty0 (4):\penalty0 322--330, August 1965.
\newblock ISSN 0029-599X.

\bibitem[Rivlin(1962)]{Rivlin1962}
TJ~Rivlin.
\newblock Polynomials of best uniform approximation to certain rational
  functions.
\newblock \emph{Numerische Mathematik}, 4\penalty0 (1):\penalty0 345--349,
  1962.

\bibitem[Rubinov(2013)]{Rubinov00}
A.~M. Rubinov.
\newblock \emph{Abstract convexity and global optimization}, volume~44.
\newblock Springer Science \& Business Media, 2013.

\bibitem[Rubinov and Simsek(1995)]{RubinovSimsek}
A.~M. Rubinov and B.~Simsek.
\newblock Conjugate quasiconvex nonnegative functions.
\newblock \emph{Optimization}, 35\penalty0 (1):\penalty0 1--22, 1995.
\newblock \doi{10.1080/02331939508844124}.

\bibitem[Singer(1997)]{singerAbstract}
Ivan. Singer.
\newblock \emph{Abstract convex analysis / Ivan Singer.}
\newblock Canadian Mathematical Society series of monographs and advanced
  texts. Wiley, New York, 1997.
\newblock ISBN 0471160156.

\bibitem[Soltan(1984)]{soltan84}
V.V. Soltan.
\newblock \emph{Introduction to the Axiomatic Theory of Convexity}.
\newblock Sthiinka, Kishinev, 1984.
\newblock In Russian.

\bibitem[van~de Vel(1993)]{vandevel93}
M.L.J. van~de Vel.
\newblock \emph{Theory of convex structures}.
\newblock North-Holland, Amsterdam New York, 1993.
\newblock ISBN 9780080933108.

\bibitem[ya~Matsushita and Xu(2016)]{Shi-yaXu}
Shin ya~Matsushita and Li~Xu.
\newblock On the finite termination of the douglas-rachford method for the
  convex feasibility problem.
\newblock \emph{Optimization}, 65\penalty0 (11):\penalty0 2037--2047, 2016.

\bibitem[Yang and Yang(2013)]{YangYang}
Yuning Yang and Qingzhi Yang.
\newblock Some modified relaxed alternating projection methods for solving the
  two-sets convex feasibility problem.
\newblock \emph{Optimization}, 62\penalty0 (4):\penalty0 509--525, 2013.

\bibitem[Zaslavski(2013)]{Zaslavski}
A.~J. Zaslavski.
\newblock Subgradient projection algorithms and approximate solutions of convex
  feasibility problems.
\newblock \emph{Journal of Optimization Theory and Applications}, 157:\penalty0
  803--819, 2013.

\bibitem[Zhao and Köbis(2018)]{Zhao}
Xiaopeng Zhao and Markus~Arthur Köbis.
\newblock On the convergence of general projection methods for solving convex
  feasibility problems with applications to the inverse problem of image
  recovery.
\newblock \emph{Optimization}, 67\penalty0 (9):\penalty0 1409--1427, 2018.

\end{thebibliography}
\end{document}